\documentclass[11pt]{article}

\usepackage{mathtools}
\usepackage{amsmath, amsthm, amsfonts,amssymb}
\usepackage{enumitem}
\usepackage{graphicx}
\usepackage{colortbl}
\usepackage{tikz}
\usepackage[utf8]{inputenc}
\usepackage{esint}
\usepackage{mathrsfs}
\usepackage{subfig,float}
\usepackage[T1]{fontenc}
\usepackage{mathrsfs}  
\usepackage{varwidth}
\usepackage{bbm} 
\usepackage{enumitem}
\usepackage{cite}
\usepackage{mathtools}
\usepackage{array}
\usepackage{graphicx}
\usepackage{caption}
\usepackage{dsfont}
\usepackage{ushort}
\usepackage{accents}

\usepackage{lineno} 

\newcommand\blfootnote[1]{%
	\begingroup
	\renewcommand\thefootnote{}\footnote{#1}%
	\addtocounter{footnote}{-1}%
	\endgroup
}

\def\d{\,\mathrm{d}}
\def\dv{\d v}

\newcommand{\lp}{\left(}
\newcommand{\rp}{\right)}

\usepackage[colorlinks=true,linkcolor=blue,citecolor=blue,urlcolor=blue,breaklinks]{hyperref}

\numberwithin{equation}{section}

\usepackage{pgfplots}
\pgfplotsset{compat = newest}

\usepackage[titletoc, toc, page]{appendix} 

\usepackage{hyperref}
\usepackage{url}
\nocite{}

\def\R{\mathbb{R}}
\def\Z{\mathbb{Z}}

\def\T{\mathbb{T}}

\def\1{\mathds{1}}

\let\mc=\mathcal

\def\d{\,\mathrm{d}}
\def\dx{\,\mathrm{d}x}

\def\ixv{\int_{\T^d} \int_{\R^d}}

\def\ix{\int_{\T^d}}

\def\p{\,\partial}

\let\eps\varepsilon

\newtheorem{thm}{Theorem}[section]

\newtheorem{lem}[thm]{Lemma}
\newtheorem{prp}[thm]{Proposition}

\theoremstyle{definition}

\theoremstyle{remark}
\newtheorem{remark}[thm]{Remark}

\hypersetup{pdftitle={KineticToNonlinearDiffusion}}
\hypersetup{pdfauthor={Josephine Evans, Daniel Morris  \& Havva Yoldaş }}

\author{Josephine Evans\footnote{Warwick Mathematics Institute, University of Warwick, Zeeman Building, Coventry CV4 7AL, United Kingdom. josephine.evans@warwick.ac.uk}
	\and	Daniel Morris\footnote{Corresponding author.} \footnotemark[3] 
	\and Havva Yolda\c{s}\footnote{Delft Institute of Applied Mathematics, Faculty of Electrical Engineering, Mathematics and Computer Science, Delft University of Technology, Mekelweg 4, 2628CD Delft, The Netherlands. d.n.l.morris@tudelft.nl \& h.yoldas@tudelft.nl}
	\footnote{Theoretical Sciences Visiting Program (TSVP), Okinawa Institute of Science and Technology Graduate University, Onna, 904-0495, Japan.}}

\title{On a class of nonlinear BGK-type kinetic equations with density dependent collision rates}

\makeindex

\begin{document}
	
	\maketitle
	
	\vspace{-10pt}

	\begin{abstract}
		
		We consider a class of nonlinear, spatially inhomogeneous kinetic equations of BGK-type with density dependent collision rates. These equations share the same superlinearity as the Boltzmann equation, and fall into the class of run and tumble equations appearing in mathematical biology. We prove that the Cauchy problem is well-posed, and the solutions propagate Maxwellian bounds over time. Moreover, we show that the solutions approach to equilibrium with an exponential rate, known as a hypocoercivity result. Lastly, we derive a class of nonlinear diffusion equations as the hydrodynamic limit of the kinetic equations in the diffusive scaling, employing both hypocoercivity and relative entropy methods. The limit equations cover a wide range of nonlinear diffusion equations including both the porous medium and the fast diffusion equations.

		
	\end{abstract}
	
	\tableofcontents
	
	\newpage
	\section{Introduction} \label{sec:intro}
	
	The purpose of this paper is to introduce and analyse a class of nonlinear kinetic equations with density dependent collision rates. These equations are of particular interest because they can be used to derive the porous medium and fast diffusion equations. Our class of equations includes those with quadratic nonlinearity that resembles the nonlinearity found in physical gas kinetic models such as the Boltzmann equation.

	The porous medium and fast diffusion equations exhibit rich mathematical structures and have been extensively studied. They admit a natural class of Lyapunov functions (such as $p$-entropies) and display a variety of complex dynamical behaviours. Depending on the parameter regime and initial data, solutions to these equations may exhibit finite speed of propagation, non-uniqueness, or even finite-time extinction, all of which pose significant analytical challenges.
	
	\medskip
	The class of equations that we study is as follows,
	\begin{align}
		\label{eq:BGK}
		\p_t f(t,x,v) + v \cdot \nabla_x f(t,x,v) = \rho_f^\alpha (t,x) \lp \rho_f (t,x) \mc M (v) - f (t,x,v)\rp
	\end{align} where $f:=f(t,x,v) \geq 0$, $(t,x,v) \in (0, +\infty] \times \T^d \times \R^d$, $d \geq 1$, $\rho_f(t,x) := \int_{\R^d}f(t,x,v)\d v$ is the spatial density, $\alpha \in \mathbb R$ and $\mc M(v) = (2\pi)^{-d/2} \exp(-|v|^2/2)$ is a Gaussian function (or the normalised Maxwellian distribution with $0$ mean velocity, unit temperature and the Boltzmann constant). Here, $\T^d$ denotes the $d-$dimensional unit torus.
	\medskip 
	
	Equation \eqref{eq:BGK} models an an ensemble of particles each of which has a position $x$ and velocity $v$. The unknown $f$ is the probability density of a "typical particle" in this ensemble over phase space. The dynamics described by \eqref{eq:BGK} are as follows. A particle will travel in a straight line with its current velocity until a random time. At this random time, a collision occurs and the particles velocity changes to one drawn from the distribution with density $\mathcal{M}(v)$ independently from everything else. The rate at which collisions happen is $\rho_f^\alpha(t,x)$. This means that the probability of a collision occurring is dependent on the number of particles in the vicinity of the colliding particle. The most natural choice of $\alpha$ in \eqref{eq:BGK} would be $\alpha=1$. In this case, the rate at which "collisions" happen is proportional to the number of particles available to be "collided" with. The probability of a collision occurring is higher when the spatial density is high.
	
	\medskip
	
	We consider a diffusive scaling $(t,x,v) \mapsto \lp \eps^2 t, \eps x, v\rp$, where $\eps >0$ is a given constant which can be taken as the dimensionless mean free path of particles. Denoting the new unknown as $f_\eps$, Equation \eqref{eq:BGK} then becomes,
	\begin{align} \label{eq:mainkinetic}
		\begin{split}
			\p_t f_\eps(t,x,v) + \frac{1}{\eps} v \cdot \nabla_x f_\eps(t,x,v) &= \frac{1}{\eps^2} \rho^\alpha_{f_\eps}(t,x) \lp \rho_{f_\eps} (t,x)  \mathcal{M} (v) - f_\eps(t,x,v)\rp\\
			f_\eps (0, x, v) &= f_{\eps,\text{in}} (x,v).
		\end{split}
	\end{align} Under the diffusive scaling, the particles undergo frequent collisions, and over a long timescale the macroscopic behaviour emerges. One of the goals of this article is to derive rigorously an equation describing this macroscopic behaviour. 
	
	\medskip 
	The main results of this article are summarised as follows. We prove 
	\begin{itemize}
		\item[\textbf{(i)}] \textbf{Global well-posedness of \eqref{eq:mainkinetic} and propagation of Maxwellian bounds.}\\
		This result relies on the maximum principle which implies that if the initial datum lies between a constant times Maxwellians, then the corresponding solutions also satisfy this property. 
		\item[\textbf{(ii)}] \textbf{Quantitative convergence to equilibrium (hypocoercivity) result as $t \rightarrow \infty$ for \eqref{eq:mainkinetic}}. \\
		Here, we use the $L^2$-hypocoercivity methods due to Dolbeault, Mouhot and Schmeiser introduced in \cite{DMS15} for mass-conserving linear kinetic equations. We adapt this technique to the nonlinear setting.
		\item[\textbf{(iii)}] \textbf{Quantitative diffusive limit of \eqref{eq:mainkinetic} as $\eps \to 0$ to the porous medium or fast diffusion equations.}\\
		Combining relative entropy and hypocoercivity methods, we show that solutions of Equation \eqref{eq:mainkinetic} as $\eps \to 0$ are functions of the form $\rho(t,x) \mc M (v)$ where $\rho(t,x)$ solves the following nonlinear diffusion equation:
		\begin{align} \label{eq:mainparabolic}
				\begin{cases}
				\begin{split}
				\partial_t \rho(t,x) &= \nabla_x \cdot \lp \rho^{-\alpha} (t,x) \nabla_x \rho (t,x) \rp,\\
				\rho(0,x) &= \rho_{\text{in}} (x).
				\end{split}
			\end{cases}
		\end{align}
	\end{itemize}
	This article forms one of very few works deriving the porous medium or fast diffusion equations from kinetic equations rigorously. To the best of our knowledge, this is the only work which treats the full range of possible $\alpha$. The limit equations \eqref{eq:mainparabolic} are well studied and appear in many modelling applications in various physical, biological, and engineering contexts. Understanding their emergence from the kinetic level is key to explaining how their exotic behaviour emerges at the macroscopic level from underlying microscopic processes.

	\paragraph{Outline of paper.} In the remainder of this section, we compare our class of equations with similar kinetic models including the BGK equation and run and tumble equations. We give a summary of the limit equations \eqref{eq:mainparabolic} and some of their interesting mathematical properties. Subsequently, we give our motivations for the present paper and we finish Section \ref{sec:intro} with the state of the art. In Section \ref{sec:ass-main_results}, we present our main results followed by some preliminary results. In Section \ref{sec:well-posedness}, we prove that the Cauchy problem \eqref{eq:mainkinetic} with initial data bounded from above and below by a constant Maxwellian is well-posed. We show the propagation of the Maxwellian bounds for all times. In Section \ref{sec:long-time-behaviour}, we establish the exponential relaxation to equilibrium in time using the hypocoercivity technique developed in \cite{DMS15}. Lastly, in Section \ref{sec:diffusive_asymptotics}, we quantify the rate of the diffusive asymptotics by combining a study of the finite-time asymptotics (Section \ref{ssec:finite_time_asymptotics}) with the uniform in $\eps$ convergence to equilibrium of the solutions to both the kinetic and the parabolic equations over long times (Section \ref{ssec:long_time_asymptotics}).
	
	\subsection{The models}
	\paragraph{The kinetic equations.}
	Nonlinear kinetic equations of the form \eqref{eq:BGK} can be considered as toy models for the Boltzmann equation as they have a kinetic structure, Maxwellian local equilibria and density dependent collision rates. 
	
	A natural comparison arises between the equations studied here and the BGK equations, which can be written as 
	\begin{align} \label{eq:BGKfull} 
		\p_t f(t,x,v) + v \cdot \nabla_x f(t,x,v) = \lambda(\rho_f(t,x)) \lp \rho_f(t,x) \mc{M}_{u_f(t,x), T_f(t,x)}(v) - f(t,x,v) \rp, 
	\end{align}
	where $\mc{M}_{u,T}(v) = (2\pi T)^{-d/2} \exp(|v-u|^2/2T)$ is the Maxwellian velocity distribution. The hydrodynamic quantities momentum $u_f$, and temperature $T_f$ are defined via
	\begin{align*}
		\rho_f(t,x)u_f(t,x) &= \int f(t,x,v)v \d v,\\
		\rho_f(t,x) (T_f(t,x) + |u_f(t,x)|^2) &= \int f(t,x,v)|v|^2 \d v.
	\end{align*} The BGK equation \eqref{eq:BGKfull} was introduced in \cite{BGK54} to serve as a computationally efficient alternative to the full Boltzmann equation while maintaining essential physical properties such as conservation of hydrodynamic quantities. In Equation \eqref{eq:BGKfull}, the collision operator, the right hand side of \eqref{eq:BGKfull}, conserves mass, momentum and kinetic energy. Our toy model \eqref{eq:BGK} simplifies \eqref{eq:BGKfull} radically to an equation whose collision operator conserves only mass. Equation \eqref{eq:BGKfull} was shown to be well posed in \cite{PP93} in the case $\lambda =1$ and the result can be straightforwardly extended to the case where $\lambda$ is bounded above and below. The original BGK equation in \cite{BGK54} has $\lambda(\rho) = \rho$. This makes the mathematical analysis more challenging. In this case, to the best of our knowledge, well-posedness is still an open problem, which is also discussed in \cite{BHP21, BP23}. In more recent mathematical literature, it is common to only consider $\lambda$ a constant, for simplicity taken generally as $\lambda=1$.
	
	Compared to the original BGK equation, the class of equations that we consider are more straightforward to study. There are two fundamental reasons for this. Firstly, the form of the collision operator means that we expect a maximum principle to hold for this class of equations. This allows us to control the nonlinearity. Secondly, the fact that we only have one collision invariant means that we can obtain, in the limit, a hydrodynamic equation which only depends on mass.

	The class of equations \eqref{eq:BGK} also falls into a large class of kinetic transport equations, called the \emph{run and tumble} equations in mathematical biology. The run and tumble equations, introduced in \cite{A80, S74}, are used in modelling the movement of bacteria that respond to chemical gradients. The collision rate $\lambda$ in \eqref{eq:BGKfull} corresponds to the tumbling rate of bacteria. Although, it is common to assume $\lambda$ to be a constant to ease the mathematical analysis, in physically more relevant cases $\lambda$ depends on the gradient of the chemoattractant density function which solves a Poisson-type equation \cite{ODA88}. On the other hand, the run and tumble equation differs from similar kinetic equations \eqref{eq:BGK} in various ways, such as the confinement mechanism and non-explicit equilibrium solutions \cite{Y23}. Under some conditions on the initial mass, it may exhibit finite time extinction of solutions \cite{BC09}, similar to its diffusive limit, Keller-Segel equation. See also \cite{EY23, EY24} for recent hypocoercivity results on various run and tumble equations. 
	
	We can also compare Equation \eqref{eq:BGK} to a similar class of kinetic Fokker-Planck type equations which are given by 
	\begin{equation} \label{eq:fp}
		\p_t f (t,x,v)+ v \cdot \nabla_x f (t,x,v)= \rho_f^\alpha (t,x)\nabla_v \cdot \lp\nabla_v f(t,x,v) + vf (t,x,v)\rp.
	\end{equation} This equation shares some similarities with the Landau equation which is a model for plasma dynamics. Equations of this type have been studied in \cite{IM21, AZ24} as part of the program to utilise De Giorgi methods in kinetic theory.

	\paragraph{The limit equations.}
	Equation \eqref{eq:mainparabolic} is known as the porous medium equation when $\alpha < 0$ and as the fast diffusion equation when $\alpha \in (0,1]$. The case where $\alpha = 0$ is the heat equation. In the case of solutions with finite mass on the whole space, these equations behave very differently in different parameter regimes. We briefly review these different behaviours as they form part of the motivation for the present and future works.
	
	For the fast diffusion equation, when $\alpha \in (0, 2/d)$, we expect the equation to be well posed and for the solutions to become smooth instantaneously. If $\alpha > 2/d$ then the diffusive effect is so large so that the equation effectively loses mass instantaneously and, thus, is ill-posed. In the critical case when $\alpha = 2/d$, the equation has non-unique solutions and extinction of solutions in finite time. We refer the reader to the lecture notes \cite{D19} for all these facts and detailed explanations. An interesting case is in dimension $d=2$ with the critical value $\alpha=1$. In this case, the equation is known as the \emph{logarithmic diffusion equation} and given by 
	\begin{align*}
		\p_t \rho (t,x)= \Delta_x \log(\rho(t,x)).
	\end{align*} The solutions to the logarithmic diffusion equation exhibit extinction in finite time, and tend to resemble these self-similar forms near extinction. A good reference for the behaviour of the logarithmic diffusion equation is \cite{VER96}.
	
	The porous medium equation is well-posed for any $\alpha <0$. However, the solutions will not be smooth. In fact if the initial data is compactly supported then this will be propagated by the equation. The solutions are then smooth inside the support and H\"older continuous at the boundary of the support \cite{V06}. The literature about the porous medium equation is huge and varied. A key reference is \cite{Vazquez-PME}. Only slightly less has been written about the fast diffusion equation, one could start with \cite{V06} and references therein. Much recent work has been done on the connection to stability of functional inequalities, see, for example \cite{BDNS23}.
	
	\paragraph{Formal diffusive limit.}
	In the subsequent sections, we justify the diffusive limit rigorously. Here, in order to better understand how this comes about we perform a formal analysis of the diffusive limit by a Hilbert series expansion. That is, forgetting issues of convergence of the series, we assume that we can expand the solution $f_{\eps}$ as a formal power series in $\eps$ over the ring of test functions. In this purely formal expansion, each coefficient corresponds to the effect of an order of $\eps$ on the solution. It turns out that only effects up to order one in $\eps$ have an effect on our solution in this case.
	
	We thus write $f_{\eps}=f_{0}+\eps f_{1}+\mc O(\eps^{2})$ where $f_{0}, f_{1}\in C^{\infty}_{c}(\T^{d}\times\R^{d})$. Here, $f_{0}$ is the part of the solution that is independent of $\eps$, and $f_{1}$ corresponds to the first order effects in $\eps$. Integrating this expansion in space, we have $\rho_{f_\eps}=\rho_{0}+\eps \rho_{1}+\mc O(\eps^{2})$ where $\rho_{i}:=\int_{\T^{d}}f_{i}\d x$ for $i=\{0,1\}$. Substituting these expansions into \eqref{eq:mainkinetic} and letting $\eps\rightarrow0$ we immediately find $f_{0}=\mathcal{M}\rho_{0}$. 
	
	Assuming the boundedness from below of $\rho_{0}$, we now expand the power law as a binomial series to obtain
	\begin{equation*}
		\rho_{f_{\eps}}^{\alpha}=\rho_{0}^{\alpha}\Bigl(1+\eps\frac{\rho_{1}}{\rho_{0}}+O(\eps^{2})\Bigr)^{\alpha}=\rho_{0}^{\alpha}+\alpha \eps\rho_{0}^{\alpha-1}\rho_{1}+ \mc O(\eps^{2}).
	\end{equation*}
	Thus, the order $\eps$ terms in \eqref{eq:mainkinetic} give
	\begin{equation*}
		v\cdot\nabla_{x}f_{0}=\rho_{0}^{\alpha}\bigl(\mathcal{M}\rho_{1}-f_{1}\bigr)+\alpha\rho_{0}^{\alpha-1}\rho_{1}\bigl(\mathcal{M}\rho_{0}-f_{0}\bigr)=\rho_{0}^{\alpha}\bigl(\mathcal{M}\rho_{1}-f_{1}\bigr),
	\end{equation*}
	so we find that $f_{1}=\mathcal{M}\rho_{1}-v\mathcal{M}\cdot\rho_{0}^{-\alpha}\nabla_{x}\rho_{0}$. Turning now to the local density of $f_{\eps}$, we integrate \eqref{eq:mainkinetic} in velocity and substitute in our formal asymptotic expansions as well as our expressions for $f_{0}$ and $f_{1}$ to reach
	\begin{equation*}
		\p_t\rho_{0}=-\nabla_{x}\cdot\int_{\R^{d}}vf_{1}\d v = -\nabla_{x}\cdot\Bigl(\rho_{1}\int_{\R^{d}}v\mathcal{M}\d v-\rho_{0}^{-\alpha}\nabla_{x}\rho_{0}\int_{\R^{d}}|v|^{2}\mathcal{M}\d v\Bigr)=\nabla_{x}\cdot\bigl(\rho_{0}^{-\alpha}\nabla_{x}\rho_{0}\bigr).
	\end{equation*}
	Thus, we formally conclude that as $\eps\rightarrow0$, $f_{\eps}$ tends to a Maxwellian distribution times a density function solving a non-linear diffusion equation, i.e.,
	\begin{equation*}
		f_{0}=\mathcal{M}\rho_{0},\qquad \partial_{t}\rho_{0}=\nabla_{x}\cdot\bigl(\rho_{0}^{-\alpha}\nabla_{x}\rho_{0}\bigr).
	\end{equation*}

	\subsection{Motivations and open questions}
	
	The goal of our current paper is to give a full study of Equation \eqref{eq:BGK} in terms of well-posedness, long-time behaviour and diffusive limit. Despite having a superlinear term, we show that this equation is extremely tractable and well behaved. Moreover, it provides a derivation of the porous medium and fast diffusion equations. While our equation does not directly come from a biological model, it is close to those used in many biological contexts.  The passage from kinetic equations to the porous medium or fast diffusion equations is not a well-studied area and our main motivation is to contribute to this.
	\medskip 
	
	We are also strongly motivated by the potential next steps: 
	\begin{itemize}
		
		\item \textbf{Hydrodynamic limits in the whole space.} We would like to study the same hydrodynamic limit and long time behaviour when the equation is posed in the whole space. The main technical challenge here is that for the physical finite mass solutions we will lose the uniform lower bound. We are particularly interested in the critical parameter regime where the limit equation is badly behaved. In particular, in dimension $d=2$, the most physical nonlinearity of $\alpha =1$ will produce a limit equation whose solutions are not unique and will become extinct in finite time. We expect the equations to be well-behaved at the level of kinetic equations. For this reason, we are interested to study the emergence of pathological behaviour in the limit. 
		
		\item \textbf{Derivation of cross diffusion equations}. We have formal calculations which show a similar class of run and tumble type equations may converge towards cross-diffusion equations. We are not aware of any results linking kinetic equations to cross diffusions in this way. Here the mathematical challenge becomes much greater. We do not only lose lower bounds, we also lose smoothness of solutions to the limit equations. 
		
		\item \textbf{Conditional results for the quadratic BGK equation}. A more tenuous motivation is to study equations which are similar to the full BGK equation with quadratic nonlinearity. The equations studied here are much simpler. However, we have some hope that they might provide some clues to the BGK equation in conditional regimes.
		
		\item \textbf{Linking long-time behaviour for the kinetic and limit equations}. Lastly, we would like to understand how hypocoercivity results for kinetic equations are influenced by their hydrodynamic limits. The fast diffusion equation is an example of a nonlinear equation with a nice entropy structure. Therefore, kinetic equations with this diffusive limit are well adapted to exploring this question.
	\end{itemize}
	
	\subsection{State of the art}
	Obtaining heuristically macroscopic PDEs as hydrodynamic limits of kinetic equations is a subject of many works. However, we are aware of only a small number of articles studying this with a rigorous limiting arguments. We summarise the articles that are closer to ours.  
	
	In a series of papers \cite{BGP87, BGPS88, BGP89}, Bardos, Golse, Perthame and Sentis study the Rosseland approximation for the radiative transfer equation which is a nonlinear kinetic transport equation similar to \eqref{eq:BGK}. This equation models the transport of photons in a starlike medium. The term corresponding to the collision rate $\lambda(\rho) :=\rho^\alpha $ in \eqref{eq:BGK} is the opacity of the medium $\sigma(T)$ at the internal energy $T$. In \cite{BGP87, BGPS88, BGP89}, the authors derive the so-called Rosseland equation as a hydrodynamic limit of the radiative transfer equation. The Rosseland equation is a degenerate parabolic equation similar to \eqref{eq:mainparabolic} with $\alpha = 1$ and a constant boundary value. Particularly \cite{BGPS88} covers partially the case corresponding to $-1<\alpha<0$ in \eqref{eq:BGK} without requiring the monotonicity of $\sigma$. They employ Schauder's fixed point argument, some compactness results and energy estimates to prove the well-posedness of the kinetic equation.
	
	In \cite{DMOS07}, Dolbeault, Markovich, Oelz and Schmeiser study a different type of kinetic transport equation with a nonlinear relaxation term towards a generalised local Gibbs state. The equation is posed in $(x,v) \in (\R^3 \times \R^3)$ with a given confining potential. In the diffusive limit, they obtain drift-diffusion equations with a nonlinear diffusion term of porous medium type, corresponding to the case $-\frac{2}{3} < \alpha <1$ in \eqref{eq:mainparabolic}. The drift term arises from the confining potential. They prove the existence and uniquness of the solutions of the kinetic equation for initial data bounded by equilibrium distributions. The diffusion limit is obtained by using a div-curl lemma based compactness argument; thus, is not quantitative. 
	
	More recently, Anceschi and Zhu, in \cite{AZ24}, study the Cauchy problem and diffusion asymptotics of a nonlinear kinetic Fokker-Planck type equation \eqref{eq:fp} where $\alpha \in [0,1]$, and $(x,v) \in \T^d \times \R^d$. They prove well-posedness and propagation of Maxwellian lower bounds of the solutions by using
	some results based on De Giorgi-Nash-Moser theory and the Harnack inequality. They derive quantitative diffusion asymptotics by combining entropic hypocoercivity, relative $\phi$-entropy, and barrier function methods. In this paper, we will follow their broad strategy for showing the quantitative diffusive limit. The key difference to our work is that the kinetic Fokker-Planck equation has smooth solutions.
	
	\paragraph{Hydrodynamic limits in kinetic theory.}
	
	Hydrodynamic limits make up a large area of active research in kinetic theory, thanks to its connection to Hilbert's sixth problem entitled \emph{"Mathematical Treatment of the Axioms of Physics"} which concerns developing rigorously the limiting processes connecting the atomistic description of the matter with the laws of motion of continua. Therefore, the most classical problem in this area is justifying the limit from the Boltzmann equation to either Euler or Navier-Stokes equations. Breakthrough results are due to Bardos, Golse and Levermore in \cite{BGL91, BGL93} and then due to Golse and Saint-Raymond in \cite{GSR04}. For a good review in this direction we refer the reader to \cite{SR09}. In \cite{SR93}, Saint-Raymond provided a complete derivation of the Navier–Stokes–Fourier equations from a BGK equation \eqref{eq:BGKfull} when $\lambda$ is a constant. A major obstacle in resolving Hilbert's sixth problem was the long-time justification of the Boltzmann equation. Recent works of Deng, Hani and Ma establishes a Boltzmann-Grad derivation valid up to the lifespan of smooth Boltzmann solutions in \cite{DHM24}, and building on this and existing hydrodynamic-limit results, derive the compressible Euler and incompressible Navier–Stokes–Fourier systems on $\T^d, \, d=\{2,3\}$ in \cite{DHM25}. These results represent the current state of the art on Hilbert's sixth problem.

	Hydrodynamic and diffusive limits are also an important area of study in kinetic theory applied to mathematical biology. Here, the limit equation is often related to a nonlinear diffusion equation. Various limits from run and tumble equations under different scalings, including the Keller-Segel equation were shown in \cite{OH00, OH02, CMPS04}. Moreover, interesting recent works about the phase-transition phenomena occurring in the Vicsek-BGK type kinetic equations where the collision operator models the alignment of agents under stochastic perturbations include \cite{DDFM20, MSW24}. See also, for example \cite{EMP23}, for limits from kinetic equations to higher order equations such as Cahn-Hilliard equation.
	
	\paragraph{Particle approximations to the porous medium or fast diffusion equations.}
	
	In contrast to the derivation from kinetic equations, there are a large number of works deriving the porous medium and fast diffusion equations directly from interacting stochastic particle systems. These studies are driven by both the motivation to understand how such systems might emerge in biological phenomena and to develop numerical methods for these equations. Deriving diffusion equations from particle approximations  remains a very active area of research dating back to \cite{O90} which is one of the first works in this area. A recent paper \cite{CEW24} contains an up to date set of references. An important research direction is the `blob-method' which uses these ideas for building numerical methods, see \cite{CCP19}.
	
	\section{Preliminary and main results}
	
	\subsection{Main results} 
	\label{sec:ass-main_results}
	We present our main results with three theorems. In what follows, we denote the weighted $L^2$ norm of a function $f$ as $\|f\|_{L^2_{x,v}(\mc M ^{-1})} := \lp \ixv |f(t,x,v)|^2 \mc M^{-1} (v)\d v \d x\rp^{\frac{1}{2}}$. 
	\begin{thm}[Well-posedness, propagation of bounds]\label{thm:well-posedness}
		Let $\eps>0, A>1, \alpha\in\R$, and suppose that $f_{\eps, \mathrm{in}}(x,v)\in L^{1}(\T^{d}\times\R^{d})$ satisfy $A^{-1}\mc M (v)\leq f_{\eps, \mathrm{in}}(x,v)\leq A \mc M (v)$. Then problem \eqref{eq:mainkinetic} with initial data $f_{\eps, \mathrm{in}}$ admits a unique weak solution $f_{\eps} (t,x,v)\in C([0,\infty);L^{1}(\T^{d}\times\R^{d}))$. Moreover, $f_{\eps}(t,x,v)$ satisfies $A^{-1} \mc M (v)\leq f_{\eps}(t, x,v)\leq A  \mc M (v)$.
	\end{thm}
	
	\begin{thm}[Long-time behaviour] \label{thm:longtime}
		Let $\eps>0, A>1, \alpha\in\R$, and let $ f_\eps(t,x,v)$ be the solution to \eqref{eq:mainkinetic} with $\|f_{\eps}\|_{L^{1}_{x,v}}=1$. Suppose that  $\rho_{f_\eps} (t,x)= \int_{\R^d}f_\eps(t,x,v)\d v$ satisfies $ A^{-1} \leq \rho_{f_\eps} (t,x) \leq A$ for all $(t,x) \in [0, \infty)\times \T^d$. Then, there exists $\eps_0 >0$ such that for all $\eps \in (0, \eps_0), t\geq0$, we have
		\begin{align} \label{thm:f_bound}
			\|f_\eps (t,x,v)-\mc M(v)\|^2_{L_{x,v}^2(\mathcal{M}^{-1})} \leq 3 e^{-\gamma t} \|f_{\eps,\mathrm{in}}(x,v)-\mc M(v)\|^2_{L_{x,v}^2(\mathcal{M}^{-1})}, 
		\end{align} where
		\begin{align*}
			\gamma = \frac{1}{8A^\alpha (2+ A^{2\alpha})}. 
		\end{align*}
	\end{thm}
	
	\begin{thm}[Diffusion asymptotics] \label{thm:asymptotics}
		Let $A>1$ and $\eps\in(0,\min\{\eps_{0},\frac{1}{2}\})$ with $\eps_{0}$ given by Theorem \ref{thm:longtime}. Consider a sequence of functions $\{f_{\eps,\mathrm{in}}\}\subset L^{1}(\T^{d}\times\R^{d})$ satisfying $A^{-1}\mc M (v)\leq f_{\eps,\mathrm{in}}(x,v)
		\leq A \mc M (v)$ and let $f_{\eps}(t,x,v)$ be the solution of \eqref{eq:mainkinetic} with initial data $f_{\eps,\mathrm{in}}(x,v)$. Let $\rho_{\mathrm{in}}(x)\in C^{2}(\T^{d})$ satisfy $A^{-1}\leq\rho_{\mathrm{in}}(x)\leq A$ and $\tilde \rho(t,x)$ be the solution of \eqref{eq:mainparabolic} with initial data $\rho_{\mathrm{in}}(x)$. If there exists some constant $\eps'\in(0,\frac{1}{2})$ such that
		\begin{equation*}
			\|f_{\eps,\mathrm{in}}(x,v)-\rho_{\mathrm{in}}(x)\mathcal{M}(v)\|_{L_{x,v}^{2}(\mathcal{M}^{-1})}\leq\eps',
		\end{equation*}
		then there exist constants $C,\gamma>0$ depending only on $\alpha, A, \|\rho_{\mathrm{in}}\|_{C^{2}}, d$ such that
		\begin{equation*}
			\|f_{\eps}(t,x,v)-\tilde \rho(t,x)\mathcal{M}(v)\|_{L^{\infty}(\R_{+};L_{x,v}^{2}(\mathcal{M}^{-1}))}\leq C(\eps+\eps')^{\gamma}.
		\end{equation*}
	\end{thm}
	
	\begin{remark}
		The condition on the initial data simply requires that the initial data for the kinetic equation and the diffusion equation remain close in the limit $\eps \rightarrow 0$. Taking $\eps':=\eps$, this is equivalent to a more classical well-preparedness condition on the initial data.
	\end{remark}
	
	\subsection{Preliminary results on hydrodynamic quantities}
	
	We introduce the hydrodynamic quantities, spatial density $\rho_\eps$, flux $j_\eps$ and energy $E_\eps$. 
	\begin{align} \label{hydro_quantitites}
		\begin{split}
			\rho_\eps(t,x) &:= \int f_\eps(t,x,v)\d v,  \\ 
			j_\eps(t,x) &:= \int f_\eps(t,x,v) v \d v , \\ 
			E_\eps(t,x) &:= \int f_\eps(t,x,v)(v \otimes v -I) \d v .  
		\end{split}
	\end{align}
	Next, we prove two results which will be needed in the subsequent sections. 
	In the rest of the paper, whenever convenient, we omit denoting the variable dependencies to keep the notation simple. 
	\begin{lem}\label{lem:flux,energy-bounds}
		Let $f_{\eps, \mathrm{in}}\in L^{1}(\T^{d}\times\R^{d})$ and let $f_{\eps}$ be the solution of \eqref{eq:mainkinetic} associated with the initial data $f_{\eps, \mathrm{in}}$. Let $\rho_\eps, j_{\eps}, E_{\eps}$ be the associated density, flux and energy as defined in \eqref{hydro_quantitites} and take arbitrary $\beta\in\R$. Then we have
		\begin{align*}
			|j_{\eps}| \leq \lp \int_{\R^{d}}\lp \frac{f_\eps}{\mc M}-\beta\rp ^{2}\mathcal{M}\d v\rp^{\frac{1}{2}} \quad \text{and} \quad 
			|E_{\eps}|\leq \lp \int_{\R^{d}}\lp \frac{f_\eps}{\mc M}-\beta\rp ^{2}\mathcal{M}\d v\rp^{\frac{1}{2}}.
		\end{align*}
	\end{lem}
	\begin{proof}
		We prove the result only for $j_{\eps}$ and the result for $E_{\eps}$ follows in a similar way. We have
		\begin{align*}
			|j_{\eps}|^{2}&= \lp \int_{\R^{d}}\lp\frac{f_\eps}{\mc M}\rp v\mc M \d v\rp^{2} =\lp \int_{\R^{d}}\lp \frac{f_\eps}{\mc M} -\beta\rp v\mathcal{M}\d v\rp ^{2}\\
			&\leq\int_{\R^{d}}|v|^{2}\mc M \d v\cdot\int_{\R^{d}}\lp \frac{f_\eps}{\mc M} -\beta\rp ^{2}\mc M \d v \\
			&=\int_{\R^{d}}\lp \frac{f_\eps}{\mc M} -\beta\rp ^{2}\mc M \d v, 
		\end{align*} 
		where the inequality above follows from Jensen's and Cauchy-Schwarz inequalities. 
	\end{proof}
	
	\begin{lem} \label{lem:hydroeqs}
		Let $f_{\eps}$ be the solution to \eqref{eq:mainkinetic} and $\rho_{\eps}, j_{\eps}, E_{\eps}$ be the hydrodynamic quantities associated to $f_{\eps}$ as defined in \eqref{hydro_quantitites}. Then,
		\begin{align}
			\partial_t \rho_\eps &= - \frac{1}{\eps}\nabla_x \cdot j_\eps, \label{eq:rho_j}\\
			\partial_t j_\eps &= - \frac{1}{\eps} \nabla_x \cdot E_\eps - \frac{1}{\eps} \nabla_x \rho_\eps - \frac{1}{\eps^2} \rho_\eps^\alpha j_\eps.  \label{eq:j_E_rho}
		\end{align}
	\end{lem}
	\begin{proof}
		This follows by integrating Equation \eqref{eq:mainkinetic}. 
	\end{proof}
	We finish this section with two remarks justifying the diffusive limit formally by looking at an argument from the moments. 
	\begin{remark}
		From Equations \eqref{eq:rho_j} and \eqref{eq:j_E_rho},  we also have
		\begin{align*}
			\partial_t (\eps \nabla_x \cdot( \rho_\eps^{-\alpha}j_\eps)) &= -\alpha \nabla_x \cdot(\rho_\eps^{-\alpha-1} j_\eps (\nabla_x \cdot j_\eps)) - \nabla_x \cdot( \rho_\eps^{-\alpha} \nabla_x \cdot E_\eps)  - \nabla_x \cdot(\rho_\eps^{-\alpha} \nabla_x \rho_\eps) - \frac{1}{\eps} \nabla_x \cdot j_\eps.
		\end{align*}
		From which, it further follows that
		\begin{align*}
			\partial_t (\rho_\eps - \eps \nabla_x \cdot( \rho_\eps^{-\alpha} j_\eps)) &= \alpha \nabla_x \cdot(\rho_\eps^{-\alpha-1} j_\eps (\nabla_x \cdot j_\eps)) + \nabla_x \cdot( \rho_\eps^{-\alpha} \nabla_x \cdot E_\eps) + \nabla_x \cdot(\rho_\eps^{-\alpha} \nabla_x \rho_\eps).
		\end{align*}
		Thus, another way to understand the diffusive limit is to assume that as $\eps \rightarrow 0$ then $\nabla_x \cdot (\rho_\eps^{-\alpha}j_\eps)$ remains bounded and $\alpha \nabla_x \cdot(\rho_\eps^{-\alpha-1} j_\eps (\nabla_x \cdot j_\eps)) + \nabla_x \cdot( \rho_\eps^{-\alpha} \nabla_x \cdot E_\eps) \rightarrow 0$. 
	\end{remark}
	
	\begin{remark} \label{rem:R_eps}
		Similarly from equations \eqref{eq:rho_j}, \eqref{eq:j_E_rho}, we can see that we expect
		$ -\frac{1}{\eps}\rho_\eps^\alpha j_\eps - \nabla_x \rho_\eps \rightarrow 0, $ as \(t \rightarrow \infty\).
	\end{remark}

	\section{Well-posedness}  \label{sec:well-posedness}
	In this section, we study the Cauchy problem \eqref{eq:mainkinetic} and show that it is well posed. To show the well-posedness of \eqref{eq:mainkinetic}, we first study a simpler Cauchy problem
	\begin{align} \label{eq:auxkinetic}
		\begin{split}
			\partial_t f + v \cdot \nabla_x f &= \lambda(\rho_{f}) \bigl(\rho_{f} \mathcal{M} - f\bigr),  \\
			f(t=0)&=f_{\text{in}}, 
		\end{split}
	\end{align}
	for $t\in\R_{+},x\in\T^{d},v\in\R^{d}$ and $\lambda:\R_{+}\rightarrow\R_{+}$. We assume that $\lambda$ is a $C^{1}$ function satisfying the following bounds,
	\begin{align} \label{hyp:lambda}
		\inf_{z\in I}\lambda(z), \; \sup_{z\in I}\bigl|\lambda'(z)\bigr|, \; \sup_{z\in I}\bigl|(z\lambda(z))'\bigr|<+\infty, \quad \text{where} \quad I:=[A^{-1},A],
	\end{align}for some $A>1$.

	Next, we prove the well-posedness of \eqref{eq:auxkinetic} by means of a fixed point argument.
	\begin{prp} \label{prp:aux_well-posedness}
		Let $A>1$, $f_{\mathrm{in}} \in L^{1}(\T^{d}\times\R^{d})$ satisfy $A^{-1}\mc M \leq f_{\mathrm{in}}\leq A \mc M $, and assume $\lambda\in C^{1}(\R_{+})$ satisfies \eqref{hyp:lambda}. Then, the problem \eqref{eq:auxkinetic} has a unique weak solution $f \in C([0,+\infty);L^{1}(\T^{d}\times\R^{d}))$, and $f$ satisfies $A^{-1}\mc M \leq f \leq A \mc M $.
	\end{prp}
	\begin{proof}
		Let $T>0$ and define the space $\mathcal{V}:=\{\phi\in C([0,T];L^{1}(\T^{d}\times\R^{d})):A^{-1}\mc M \leq \phi \leq A \mc M \}$ with norm $\|\phi\|_{\mc V}:=\sup_{0\leq t\leq T}\|\phi(t,\cdot,\cdot)\|_{L^{1}}$. For $f\in\mathcal{V}$, define $\Gamma(f)=g$ as the solution of the linear transport problem
		\begin{align*}
			\partial_{t}g+v\cdot\nabla_{x}g &= \lambda(\rho_{f})(\rho_{f}\mathcal{M}-g), \\
			g(t=0)&=f_{\text{in}},
		\end{align*}
		obtained via the method of characteristics. We first show that $\Gamma$ maps $\mathcal{V}$ into $\mathcal{V}$. Set $r=A\mathcal{M}-g$; then $r$ solves
		\begin{align*}
			\partial_{t}r+v\cdot\nabla_{x}r +\lambda(\rho_{f})r &= \lambda(\rho_{f})(A-\rho_{f})\mathcal{M}, \\
			r(t=0)&=A\mathcal{M}-f_{\text{in}}.
		\end{align*}
		By the method of characteristics we may explicitly solve for $r$ and use $r(t=0)\geq0$, $\rho_{f}\leq A$ to show that $r\geq0$, and so $g\leq A\mathcal{M}$. Similarly, taking $r=g-A^{-1}\mathcal{M}$, we find that $A^{-1}\mathcal{M}\leq g$ and so $g\in \mathcal{V}$. 
		
		Let us now show that $\Gamma$ is a contraction on $\mathcal{V}$. To this end, we let $f_{1},f_{2}\in\mathcal{V}$, set $g_{1}=\Gamma(f_{1}), g_{2}=\Gamma(f_{2})$ and $w=g_{1}-g_{2}$. Then $w$ satisfies 
		\begin{align*}
			\partial_{t}w+v\cdot\nabla_{x}w&=\lambda(\rho_{f_{1}})\rho_{f_{1}}\mathcal{M}-\lambda(\rho_{f_{1}})g_{1}-\lambda(\rho_{f_{2}})\rho_{f_{2}}\mathcal{M}+\lambda(\rho_{f_{2}})g_{2} \\
			&=\lp \lambda(\rho_{f_{1}})\rho_{f_{1}}-\lambda(\rho_{f_{2}})\rho_{f_{2}}\rp\mathcal{M} -g_{2}\bigl(\lambda(\rho_{f_{1}})-\lambda(\rho_{f_{2}})\bigr)-\lambda(\rho_{f_{1}})w.
		\end{align*}
		Collecting all terms involving $w$ on the left-hand-side, we find
		\begin{equation*}
			\partial_{t}w+v\cdot\nabla_{x}w +\lambda(\rho_{f_{1}})w = \bigl(\rho_{f_{1}}\lambda(\rho_{f_{1}})-\rho_{f_{2}}\lambda(\rho_{f_{2}})\bigr)\mc M -g_{2}\bigl(\lambda(\rho_{f_{1}})-\lambda(\rho_{f_{2}})\bigr):=S.
		\end{equation*}
		Multiplying by $\text{sgn}(w)$ and integrating in $x$ and $v$ we find that, for $\lambda_{0}:=\inf_{z\in[A^{-1},A]}|\lambda(z)|$,
		\begin{equation*}
			\frac{\d}{\d t}\|w(t,\cdot,\cdot)\|_{L^{1}} + \lambda_{0}\|w(t,\cdot,\cdot)\|_{L^{1}}\leq\|S(t,\cdot,\cdot)\|_{L^{1}}.
		\end{equation*}
		Using the fact that $g_{2}\in\mathcal{V}$ and so is bounded from above by $A\mathcal{M}$, we find
		\begin{align*} 
			\begin{split}
				\| S(t,\cdot,\cdot) \|_{L^{1}} &\leq\| \mathcal{M} \|_{\infty}\bigl(\sup_{z\in I}|(z\lambda(z))'|+A\sup_{z\in I}|\lambda'(z)|\bigr)\|f_{1}(t,\cdot,\cdot)-f_{2}(t,\cdot,\cdot)\|_{L^{1}} \\ &:=C\|f_{1}(t,\cdot,\cdot)-f_{2}(t,\cdot,\cdot)\|_{L^{1}}.
			\end{split}
		\end{align*}
		Thus, using $w(t=0)=0$, by Gr\"onwall's lemma we find that 
		\begin{align*}
			\|w(t,\cdot,\cdot)\|_{L^{1}}\leq C\int_{0}^{t} e^{-\lambda_{0}(t-s)}\|f_{1}(s,\cdot,\cdot)-f_{2}(s,\cdot,\cdot)\|_{L^{1}} \d s \leq \frac{C}{\lambda_{0}}\bigl(1-e^{-\lambda_{0}t}\bigr)\|f_{1}-f_{2}\|_{\mathcal{V}},
		\end{align*}
		and so we conclude that
		\begin{equation*}
			\|\Gamma(f_{1})-\Gamma(f_{2})\|_{\mathcal{V}}\leq\frac{C}{\lambda_{0}}\bigl(1-e^{-\lambda_{0}T}\bigr)\|f_{1}-f_{2}\|_{\mathcal{V}}.
		\end{equation*}
		Therefore, for $T$ sufficiently small, $\Gamma$ is indeed a contraction on $\mathcal{V}$ and Banach's fixed point theorem gives the existence of a unique fixed point $g$ of $\Gamma$ in $\mathcal{V}$. It can be easily seen that $g$ is a solution of \eqref{eq:auxkinetic} on $[0,T]$ and so has constant $L^{1}$ norm in time. 
		
		As $T$ depends only on $\lambda$ and $A$, and the fixed point $g$ satisfies the same bounds as $f_{\mathrm{in}}$, we may iterate this process to extend $g$ to all finite times. Since $g$ has constant $L^{1}$ norm and is uniformly bounded from above and below for all times we conclude that $g\in C([0,\infty);L^{1}(\T^{d}\times\R^{d}))$ with $A^{-1}\mc M \leq g \leq A \mc M $ as claimed.
	\end{proof}
	With this, the well-posedness of \eqref{eq:mainkinetic} immediately follows.
	
	\begin{proof}[Proof of Theorem \ref{thm:well-posedness}]
		Set $\lambda(z)=z^{\alpha}$. Clearly, such $\lambda$ satisfies \eqref{hyp:lambda} for all $\alpha\in\R$ and so we may use Proposition \ref{prp:aux_well-posedness} to obtain the existence of a unique solution $f\in C([0,+\infty);L^{1}(\T^{d}\times\R^{d}))$ of the unscaled equation \eqref{eq:BGK}. Defining $f_{\eps}(t,x,v):=f(\eps^{2}t,\eps x,v)$ as the parabolic rescaling of $f$ it follows that $f_{\eps}$ is the unique solution of \eqref{eq:mainkinetic}.
	\end{proof}

	\section{Long-time behaviour} \label{sec:long-time-behaviour}
	
	In this section, we show that the solution to Equation \eqref{eq:mainkinetic} converges to equilibrium as $t \rightarrow \infty$, in the space $L_{x,v}^2(\mathcal{M}^{-1}):= \left \{ g \, : \, \int_{\T^d} \int_{\R^d} \frac{|g (t, x,v)|^2}{\mc M (v)} \d v \d x  < \infty\right \}$. 
	We achieve this by adapting the hypocoercivity techniques developed in \cite{DMS15}. Particularly, we have to pay attention to ensure that the nonlinear terms do not cause us problems.

	\subsection{ $L^2$-hypocoercivity \`a la Dolbeault-Mouhot-Schmeiser}
	For our theorem in this section, we need some straightforward technical results. We collect them in the lemma below. 
	
	We consider $u(x), \, x\in \T^d$ and define its Fourier transform $\hat u (k) := \int_{\T^d} u(x) e^{-2\pi i k x} \d x$. We also define the operator  $(I-\Delta_x)^{-1}$ for $x \in \T^d$ by $\mathscr{F}((I-\Delta_x)^{-1} u(x)) := \frac{1}{1+4\pi^2 |k|^2} \hat{u}(k)$. 
	
	\begin{lem} \label{lem:Fourier_est}
		The following estimates hold for  $(I-\Delta_x)^{-1}$:
		\begin{align}
			\| (I-\Delta_x)^{-1}\partial^2_{x_i, x_j} u\|_{L^2_x}^2 &\leq  \|u\|_{L^2_x}^2, \label{eq:grad^2}\\
			\| (I-\Delta_x)^{-1}\partial_{x_i} u\|_{L^2_x}^2 &\leq  \|u - \bar{u}\|_{L^2_x}^2 \leq  \|u\|_{L^{2}_{x}}^2, \label{eq:grad}\\
			\| (I-\Delta_x)^{-1/2} \nabla_x u\|_{L^2_x}^2 &  \geq C\|u - \bar{u}\|_{L^2_x}^2, \label{eq:grad_half}
		\end{align} where $\bar u := \int_{\T^{d}} u(x) \d x$ and $C= 4\pi^2/(1+ 4 \pi^2) $.
	\end{lem} 
	\begin{proof}
		Notice that the operator  $(I-\Delta_x)^{-1}$ commutes with taking  derivatives of $u$ in $x$, i.e., $\nabla_x ((I-\Delta_x)^{-1} u )= (I-\Delta_x)^{-1} \nabla_x u$. It is also positive and self adjoint and so has a square root which is given by  $\mathscr{F}((I-\Delta_x)^{-1/2} u) = \frac{1}{(1+4\pi^2 |k|^2)^{1/2}} \hat{u}$. 
	By Plancherel theorem, we have $\|u\|^2_{L^2_x} = \sum_{k \in \Z^d} |\hat u (k)|$. Moreover, for the zero-th Fourier mode, i.e., $k=0$, we have that $\hat u (0) = \int_{\T^{d}} u(x) \d x  = \bar u$.
	This means
		\begin{align*}
			\|u - \bar u \|^2_{L^2_x} = \sum_{k\in \Z^d, k\neq 0} |\hat u (k)|^2 \leq \sum_{k\in \Z^d} | \hat u (k)|^2 = \|u\|^2_{L^2_x}.
		\end{align*}
	The rest follows by direct computation. Then, we have
		\begin{align*}
			\| (I-\Delta_x)^{-1}\partial^2_{x_i, x_j} u\|_{L^2_x}^2 &= \sum_{k \in \mathbb{Z}^d} \left| \frac{ -4 \pi^2 k_i k_j}{1+4\pi^2 |k|^2} \hat{u}(k)\right|^2   \leq \sum_{k \in \mathbb{Z}^d} |\hat{u}(k)|^2 = \|u\|_{L^2_x}^2, \\
			\| (I-\Delta_x)^{-1}\partial_{x_i} u\|_{L^2_x}^2 &= \sum_{k \in \mathbb{Z}^d} \left| \frac{2 \pi i k_i}{1+4\pi^2 |k|^2} \hat{u}(k)\right|^2   \leq \sum_{k \in \mathbb{Z}^d, k \neq 0} |\hat{u}(k)|^2 = \|u - \bar{u}\|_{L^2_x}^2 \leq  \|u\|_{L^2_x}^2,\\
			\| (I-\Delta_x)^{-1/2} \nabla_x u\|_{L^2_x}^2 &= \sum_{k \in \mathbb{Z}^{d}} \frac{4 \pi^2|k|^2}{1+ 4\pi^2 |k|^2} |\hat{u}(k)|^2  \geq C \sum_{k \in \mathbb{Z}^{d}, k \neq 0} |\hat{u}(k)|^2  = C\|u - \bar{u}\|_{L^2_x}^2, 
		\end{align*} where $C= 4\pi^2/(1+ 4 \pi^2)$.
	\end{proof}
	
	\begin{proof}[Proof of Theorem \ref{thm:longtime}]
		To prove the theorem, we adapt the proof in $L^2$ hypocoercivity
		in \cite{DMS15}. We are able to apply it fairly directly in the nonlinear setting by using the upper and lower bounds on $\rho$. We present it in a slightly different form and make some simplifications.
		
		The strategy of showing hypocoercivity in \cite{DMS15} involves creating a new norm which is equivalent to the weighted $L^2$ norm. This is done by adding a small bounded perturbation to the $L^2$ norm which depends only on the hydrodynamic quantities. On this perturbed norm we are able to construct a Gr\"onwall argument.
		
		First we look at the dissipation of the $L^2$ norm. 
		Occasionally we use the shorter notations $f_\eps (v) = f_\eps (t,x,v)$ or $f_\eps (t) = f_\eps (t,x,v)$. We first note that,
		\begin{equation*}
			\frac{\d}{\d t}\ixv \frac{(f_\eps-\mc M)^2 }{\mc M } \dv \dx=2\ixv \frac{f_\eps-\mc M}{\mc M }\partial_{t}f_{\eps} \dv \dx = 2\ixv \frac{f_\eps}{\mc M }\partial_{t}f_{\eps} \dv \dx,
		\end{equation*}
		where the second term vanishes due to the conservation of mass. We then have, 
		\begin{align*}
			\frac{\d}{\d t}\ixv &\frac{(f_\eps-\mc M)^2 }{\mc M } \dv \dx 
			\\ &= - \frac{2}{\eps} \ixv \frac{f_\eps}{
				\mc M}  (v \cdot \nabla_x f_\eps) \dv \dx + \frac{2}{\eps^2} \ixv \frac{f_\eps}{
				\mc M }  \rho_{\eps}^{\alpha} (\rho_{\eps} \mc M - f_\eps ) \dv \dx \\
			&= -\frac{1}{\eps}\ixv  v\cdot  \frac{\nabla_x(f_\eps^2) }{\mc M} \dv \dx + \frac{2}{\eps^2} \int_{\T^d} \rho_\eps^{\alpha} \int_{\R^d} \int_{\R^d} \lp  f_\eps(u)f_\eps(v) -f_\eps^2(v) \frac{\mc M(u)}{\mc M(v)} \rp \dv \d u \dx \\
			&= -\frac{1}{\eps^2} \int_{\T^d} \rho_\eps^{\alpha}  \int_{\R^d} \int_{\R^d} \lp \lp \frac{f_\eps(u)}{\mc M(u)} - \frac{f_\eps(v)}{\mc M(v)} \rp^2 + \frac{f_\eps^2(u)}{\mc M^2 (u)} - \frac{f_\eps^{2}(v)}{\mc M^2 (v)} \rp \mc M (u) \mc M (v) \d u \dv \dx \\
			& = -\frac{1}{\eps^2} \int_{\T^d} \rho_\eps^{\alpha}  \int_{\R^d} \int_{\R^d}
			\lp\frac{f_\eps(u)}{\mc M(u)} - \frac{f_\eps(v)}{\mc M(v)} \rp^2
			\mc M (u) \mc M (v) \d u \dv \dx.
		\end{align*} 
		
		If we write $g_\eps := f_\eps - \rho_\eps \mc M $, then $\int g_\eps \d v = 0$. So, we have
		\begin{align}
			\frac{\d}{\d t} \ixv\frac{(f_\eps-\mc M)^2}{\mc M} \dv \dx   &= - \frac{1}{\eps^2}\int_{\T^d} \rho_\eps^{\alpha} \int_{\R^d} \int_{\R^d} \lp g_\eps^2(u) \frac{\mc M (v)}{\mc M (u)} + g_\eps^2(v) \frac{\mc M (u)}{\mc M (v)}  - 2g_\eps(v) g_\eps(u) 
			\rp \d u \dv \dx \nonumber \\
			& = -\frac{2}{\eps^2} \int_{\T^d} \rho_\eps^\alpha \int_{\R^d} \frac{g_\eps^2}{\mc M } \dv \dx\nonumber
			\\&= -\frac{2}{\eps^2} \int_{\mathbb{T}^d} \rho_\eps^\alpha  \nonumber \int_{\R^d}\frac{((I-\Pi)f_\eps)^2}{\mc M} \dv \dx \nonumber \\
			& \leq - \frac{2}{\eps^2} \frac{1}{A^{\alpha}} \|(I-\Pi)f_\eps(t)\|^2_{L_{x,v}^2(\mc M^{-1})},  \label{ineq:f_2}
		\end{align} where $(I-\Pi) f := f - \rho \mc M$.
		
		In order to construct the perturbation term, we use the hydrodynamic quantities defined in \eqref{hydro_quantitites}. We recall the equations on the hydrodynamic quantities (given earlier by Lemma \ref{lem:hydroeqs}) below, so that we can follow the subsequent computations with ease.
		\begin{align*}
			\partial_t \rho_\eps &= - \frac{1}{\eps}\nabla_x \cdot j_\eps,\\
			\partial_t j_\eps &= - \frac{1}{\eps} \nabla_x \cdot E_\eps - \frac{1}{\eps} \nabla_x \rho_\eps - \frac{1}{\eps^2} \rho_\eps^\alpha j_\eps. 
		\end{align*}
		We then take a fairly standard perturbation term, now classical in hypocoercivity theory:
		\begin{align*}
			\frac{\d}{\d t} \ix j_\eps \cdot &(I-\Delta_x)^{-1} \nabla_x \rho_\eps \dx \\& = -\frac{1}{\eps}\int_{\T^d} j_\eps \cdot (I-\Delta_x)^{-1}\nabla_x (\nabla_x \cdot j_\eps) \dx - \frac{1}{\eps} \int_{\T^d} \nabla_x \cdot E_\eps (I-\Delta_x)^{-1} \nabla_x \rho_\eps \dx \\
			&\qquad \quad - \frac{1}{\eps} \int_{\T^d} \nabla_x \rho_\eps \cdot (I-\Delta_x)^{-1} \nabla_x \rho_\eps \dx  - \frac{1}{\eps^2} \int_{\T^d}\rho_\eps^\alpha j_\eps \cdot (I-\Delta_x)^{-1} \nabla_x \rho_\eps \dx  \\
			& \leq \frac{1}{\eps} \lp  \|j_\eps\|_{L^2_x} \|(I-\Delta_x)^{-1} \nabla_x \lp \nabla_x\cdot j_\eps \rp \|_{L^2_x} + \|E_\eps\|_{L^2_x} \|(I-\Delta_x)^{-1} \nabla^2_x \rho_\eps \|_{L^2_x} \rp  \\
			& \qquad \quad + \frac{1}{\eps^2} \|\rho_\eps^\alpha\|_{L^\infty} \|j_\eps\|_{L^2_x} \|(I-\Delta_x)^{-1} \nabla_x \rho_\eps\|_{L^2_x}  - \frac{1}{\eps} \|(I-\Delta_x)^{-1/2} \nabla_x \rho_\eps\|_{L^2_x}^2 
		\end{align*} Now using Lemma \ref{lem:Fourier_est} and Young's inequality we estimate the above quantity further as,
		\begin{align}
			\frac{\d}{\d t} \ix j_\eps \cdot (I-\Delta_x)^{-1} &\nabla_x \rho_\eps \dx \nonumber
			\\ &\leq \frac{1}{\eps} \lp \|j_\eps\|_{L^2_x}^2 + \|E_\eps\|_{L^2_x}^2 + \frac{1}{4} \|\rho_\eps-\bar{\rho_\eps}\|_{L^2_x}^2 \rp \nonumber
			\\ & \qquad  +\frac{1}{\eps^2} \|\rho_\eps^\alpha\|_{L^\infty} \lp \frac{\|\rho_\eps^\alpha\|_{L^\infty}}{\eps} \|j_\eps\|_{L^2_x}^2 + \frac{\eps}{4 \|\rho_\eps^\alpha\|_{L^\infty}}\|\rho_\eps- \bar{\rho_\eps}\|_{L^2_x}^2 \rp - \frac{C}{\eps}\|\rho_\eps-\bar{\rho_\eps}\|_{L^2_x}^2 \nonumber \\
			& = - \frac{1}{\eps}\bigl(C-\frac{1}{2}\bigl) \|\rho_\eps - \bar{\rho_\eps}\|_{L^2_x}^2 + \frac{1}{\eps} \|E_\eps\|_{L^2_x}^2 +\frac{1}{\eps}\|j_{\eps}\|_{L^{2}_{x}}^{2}+  \frac{1}{\eps^3}\|\rho_\eps^\alpha\|_{L^\infty}^2 \|j_\eps\|_{L^2_x}^2 \nonumber \\
			& \leq - \frac{1}{4\eps} \|\rho_\eps - \bar{\rho_\eps}\|_{L^2_x}^2 +  \frac{1}{\eps^3}(2 + \|\rho_\eps^\alpha\|_{L^\infty}^2) \|(I-\Pi) f_\eps\|^2_{L^2_{x,v}(\mc M^{-1})}, \label{ineq:j_grad_rho}
		\end{align} for $\eps\leq1$ where the last inequality follows from applying Lemma \ref{lem:flux,energy-bounds} with $\beta=\rho_{\eps}$ and then integrating in space. To ease notation, we also bounded the constant appearing in Lemma \ref{lem:Fourier_est} from below, $C=4\pi^{2}/(1+4\pi^{2})>3/4$.
	
		Then, using \eqref{ineq:f_2} and \eqref{ineq:j_grad_rho}, we obtain
		\begin{multline} \label{eq:gronwall}
			\frac{\d}{\d t} \lp \|f_\eps-\mc M\|^2_{L_{x,v}^2(\mc M^{-1})} + \frac{\eps}{A^{\alpha}(2 + A^{2\alpha})}\int_{\T^d} j_\eps \cdot (I-\Delta_x)^{-1} \nabla_x \rho_\eps \dx\rp \\
			\leq - \frac{1}{\eps^2} \frac{1}{A^{\alpha}} \|(I-\Pi) f_\eps\|^2_{L_{x,v}^2(\mc M ^{-1})} - \frac{1}{4 A^{\alpha} (2+A^{2\alpha})} \|\rho_\eps-\bar{\rho_\eps}\|_{L^2_x}^2.
		\end{multline}
		For $\eps$ sufficiently small, we have
		\begin{align*} \frac{\eps}{A^{\alpha}(2+ A^{2\alpha})}\int_{\T^d} j_\eps \cdot (I-\Delta_x)^{-1} \nabla_x \rho_\eps \dx  &\leq \frac{\eps}{A^{\alpha}(2+ A^{2\alpha})} \|j_\eps\|_{L^2_x} \|(I-\Delta_x)^{-1} \nabla_x \rho_\eps\|_{L^2_x}\\
			& \leq \frac{\eps}{A^{\alpha}(2+ A^{2\alpha})} \|f_\eps-\mc M\|^2_{L_{x,v}^2(\mc M^{-1})} \leq \frac{1}{2} \|f_\eps-\mc M\|^2_{L_{x,v}^2(\mc M^{-1})},  \end{align*} with the second equality following from Lemma \ref{lem:flux,energy-bounds} with $\beta=1$. Consequently for $\eps$ sufficiently small, we obtain
		\begin{align*}
			\frac{\d}{\d t } \Big ( &\|f_\eps-\mc M\|^2_{L^2(\mc M^{-1})}  + \frac{\eps}{A^{\alpha}(2+ A^{2\alpha})}\int_{\T^d} j_\eps \cdot (I-\Delta_x)^{-1} \nabla_x \rho_\eps \dx \Big )   \\
			&\leq - \frac{1}{4A^{\alpha} (2+ A^{2\alpha})} \bigl(\|(I-\Pi) f_\eps\|^2_{L_{x,v}^2(\mc M ^{-1})}+\|\rho_\eps-\bar{\rho_\eps}\|_{L^2_x}^2\bigr)\\
			&= - \frac{1}{4A^{\alpha} (2+ A^{2\alpha})} \|f_\eps-\mc M\|^2_{L_{x,v}^2(\mc M^{-1})}\\
			& \leq -\frac{1}{8A^{\alpha} (2+ A^{2\alpha})} \lp \|f_\eps-\mc M\|^2_{L_{x,v}^2(\mc M^{-1})} + \frac{\eps}{A^{\alpha}(2+ A^{2\alpha})}\int_{\T^d} j_\eps \cdot (I-\Delta_x)^{-1} \nabla_x \rho_\eps \dx \rp.
		\end{align*}
		So by Gr\"onwall's lemma, we have
		\begin{multline*}  \|f_\eps(t)-\mc M\|^2_{L_{x,v}^2(\mc M^{-1})} + \frac{\eps}{A^{\alpha}(2+ A^{2\alpha})}\int_{\T^d} j_\eps(t) \cdot (I-\Delta_x)^{-1} \nabla_x \rho_\eps(t) \dx \\
			\leq e^{-\gamma t} \lp \|f_\eps(0)-\mc M\|^2_{L_{x,v}^2(\mc M^{-1})} + \frac{\eps}{A^{\alpha}(2+ A^{2\alpha})}\int_{\T^d} j_\eps(0) \cdot (I-\Delta_x)^{-1} \nabla_x \rho_\eps(0) \dx \rp, 
		\end{multline*} from which \eqref{thm:f_bound} follows. 
	\end{proof}

	\section{Diffusive asymptotics}
	\label{sec:diffusive_asymptotics}
	
	In this section, we quantify the rate of the diffusive asymptotics by combining a study of the finite-time asymptotics with the uniform in $\eps$ convergence to equilibrium of both the kinetic and the parabolic solutions over long times. The long time asymptotics follow from the exponential relaxation to equilibrium of both the kinetic and the parabolic solutions, while the finite-time asymptotics require a delicate study of the relative entropy between the solutions.
	
	Let us begin by quickly summarising the long time asymptotics.
	\subsection{Long-time asymptotics}
	\label{ssec:long_time_asymptotics}
	We first state well-posedness and regularity result on the solution to the nonlinear diffusion equation. 
	\begin{lem} \label{lem:rho_well-posedness}
		Let $A>1$ and $\rho_{\mathrm{in}}\in C^{k}(\T^{d})$ for $k\geq0$ with $A^{-1}\leq\rho_{\mathrm{in}}\leq A$. Then there exists a unique classical solution $\tilde \rho\in C^{\infty}((0,\infty)\times\T^{d})\cap C^{k}([0,\infty)\times\T^{d})$ of the problem \eqref{eq:mainparabolic} with initial data $\rho_{\mathrm{in}}$, and $\tilde \rho$ satisfies $A^{-1}\leq \tilde \rho \leq A$.
	\end{lem}
	Notice that Equation \eqref{eq:mainparabolic} is uniformly parabolic away from zero and $+\infty$, so the classical theory applies for initial data uniformly bounded from above and below. We then have well-posedness and a strong maximum principle and Lemma \ref{lem:rho_well-posedness} follows. We refer the reader to \cite{Vazquez-PME} for a review of the theory of the nonlinear diffusion equation.
	\medskip
	
	We are now ready to state the long-time asymptotics. We obtain this result combining the hypocoercivity result for \eqref{eq:mainkinetic} provided by Theorem \ref{thm:longtime} and the exponential relaxation to equilibrium of the solution to the parabolic equation \eqref{eq:mainparabolic} on the torus due to the Poincar\'{e} inequality. 
	\begin{prp}\label{prp:long_time_asymptotics}
		Let $f_{\eps, \mathrm{in}}\in L^{1}(\T^{d}\times\R^{d})$  and $f_{\eps}$ be the solution of Equation \eqref{eq:mainkinetic} with initial data $f_{\eps, \mathrm{in}}$ for $\eps\in(0,\eps_{0})$ with $\eps_{0}>0$ given by Theorem \ref{thm:longtime}. Let $\rho_{\mathrm{in}}\in C^{2}(\T^{d})$ satisfy $A^{-1}\leq\rho_{\mathrm{in}}\leq A$ and $\tilde \rho$ be the solution of \eqref{eq:mainparabolic} with initial data $\rho_{\mathrm{in}}$. Then there exist positive constants $C,\gamma>0$ depending only on $\alpha,A,d$ such that for any $t\geq0$ we have
		\begin{equation*}
			\|f_{\eps}-\tilde \rho\mathcal{M}\|_{L^{2}_{x,v}(\mathcal{M}^{-1})}\leq Ce^{-\gamma t}.
		\end{equation*}
	\end{prp}
	\begin{proof}
		First, we set $1=\int_{\mathbb{T}^d \times \mathbb{R}^d} f_{\eps, \mathrm{in}}\dx \dv = \int_{\mathbb{T}^d} \rho_{\mathrm{in}} \dx$ and note that the masses of $\tilde \rho$ and of $f_\eps$ are conserved in time.
		
		Then, we have
		\begin{equation*}
			\frac{1}{2}\frac{\d}{\d t}\|\tilde \rho -1\|_{L_{x}^{2}}^{2}=\int_{\T^{d}}(\Tilde{\rho}-1)\partial_{t}\Tilde{\rho}\d x=\int_{\T^{d}}\Tilde{\rho}\partial_{t}\Tilde{\rho}\d x=\int_{\T^{d}}\Tilde{\rho}\nabla_{x}\cdot\bigl(\Tilde{\rho}^{-\alpha}\nabla_{x}\Tilde{\rho}\bigr)\d x.
		\end{equation*}
		where the second inequality follows from the conservation of mass. Integrating by parts, we then find
		\begin{equation*}
			\frac{1}{2}\frac{\d}{\d t}\|\tilde \rho -1\|_{L_{x}^{2}}^{2}=-\int_{\T^{d}}  \tilde \rho ^{-\alpha}|\nabla_{x}\tilde \rho|^{2}\d x\leq-C\|\nabla_{x}\tilde \rho\|_{L_{x}^{2}}^{2}=-C\|\nabla_{x}(\tilde \rho-1)\|_{L_{x}^{2}}^{2} \leq -\tilde{C} \|\tilde \rho-1\|^2_{L^2_x},
		\end{equation*}
		for some $C>0$ depending only on $\alpha, A$ where the last inequality follows from Poincar\'e inequality and $\tilde{C}$ only depends on $\alpha,A, d$. Then, Gr\"onwall's lemma yields
		\begin{equation*}
			\|\tilde \rho-1\|_{L_{x}^{2}}\leq e^{-2\tilde{C}t}\|\rho_{\mathrm{in}}-1\|_{L_{x}^{2}}.
		\end{equation*}
		Then thanks to \eqref{thm:longtime}, there exists $C'>0$ depending only on $\alpha, A$ such that
		\begin{equation*}
			\|f_{\eps}-\mathcal{M}\|_{L_{x,v}^{2}(\mathcal{M}^{-1})}\leq\|f_{\eps}\|_{L_{x,v}^{2}(\mathcal{M}^{-1})}\leq e^{-C't}\|f_{\eps,\mathrm{in}}-1\|_{L_{x,v}^{2}(\mathcal{M}^{-1})},
		\end{equation*}
		and so we find that
		\begin{align*} 
			\|f_{\eps}-\tilde \rho\mathcal{M}\|_{L_{x,v}^{2}(\mathcal{M}^{-1})}&\leq\|f_{\eps}-\mathcal{M}\|_{L_{x,v}^{2}(\mathcal{M}^{-1})}+\| \tilde \rho-1\|_{L_{x}^{2}}\nonumber \\
			&\lesssim e^{-\gamma t},
		\end{align*}
		for some $\gamma>0$ depending only on $\alpha, A$ and $d$.
		
	\end{proof}
	Thus, we may control the distance between the solution to the kinetic equation \eqref{eq:mainkinetic} and the solution to the parabolic equation \eqref{eq:mainparabolic} in the long-time by the sum of their respective exponential rates of convergence to equilibrium.
	
	\subsection{Finite-time asymptotics}
	\label{ssec:finite_time_asymptotics}
	Let us now consider the finite-time diffusive asymptotics. In order to simplify the exposition we work with the rescaled solution $h_{\eps} := f_{\eps}/\mathcal{M}$ which solves
	\begin{align} \label{eq:h}
		\begin{split}
			\partial_t h_{\eps} + \frac{1}{\eps}v\cdot \nabla_x h_{\eps} &= \frac{1}{\eps^2} \rho_{\eps}^{\alpha}\lp \rho_{\eps}-h_{\eps}\rp, \\
			h(t = 0) &= h_{\mathrm{in}} := f_{\eps, \mathrm{in}}/\mathcal{M}.
		\end{split}
	\end{align}
	The main result of this subsection is the following.
	\begin{prp} \label{prp:short-time}
		Let $\rho_{\mathrm{in}}\in C^{2}(\T^{d})$ and $h_{\mathrm{in}}\in L^{1}_{x,v}(\T^{d}\times\R^{d})$ both be valued in $[A^{-1},A]$ for $A>1$, and take $\eps\in(0,\min\{\eps_{0},1\})$ with $\eps_{0}$ given by Theorem \ref{thm:longtime}. Let $h_{\eps}$ be the solution of \eqref{eq:h} with initial data $h_{\mathrm{in}}$ and $\tilde \rho$ the solution of \eqref{eq:mainparabolic} with initial data $\rho_{\mathrm{in}}$. Then, there exist constants $C, C'>0$ depending only on $\alpha, A, \|\rho_{\mathrm{in}}\|_{C^{2}}$ such that for all $t\geq0$, the following estimate holds,
		\begin{equation}
			\|h_{\eps}-\tilde \rho\|_{L_{x,v}^{2}(\mathcal{M})}^{2}\leq e^{Ct}\Bigl(\|h_{\mathrm{in}}-\rho_{\mathrm{in}}\|_{L_{x,v}^{2}(\mathcal{M})}^{2}+C'\eps\bigl(1+t^{\frac{1}{2}}\bigr)\Bigr).
		\end{equation}
	\end{prp}
	\noindent This result relies on the uniform bounds from above and below of the densities $\rho_{\eps},\tilde \rho$ and the regularity of the limit equation to write a Gr\"onwall argument on the distance $\|f_{\eps}-\tilde \rho\mathcal{M}\|_{L^{2}_{x,v}(\mathcal{M}^{-1})}$. In order to bound the remaining terms by $\eps$, we then use the quantified rate of convergence to zero of the time integrated norms of the flux and the energy. This rate is uniform in $\eps$ on finite times and follows from Theorem \ref{thm:longtime} and Lemma \ref{lem:flux,energy-bounds}.
	
	\begin{proof}
		We begin by calculating the dissipation of the relative entropy,
		\begin{align*}
			\frac{1}{2}\frac{\d}{\d t}\|h_{\eps}-\tilde \rho\|_{L^{2}_{x,v}(\mathcal{M})}^{2}&=\int_{\T^{d}}\int_{\R^{d}}h_{\eps}\partial_{t}h_{\eps}\mathcal{M}\d v\d x-\int_{\T^{d}}\bigl(\partial_{t}\rho_{\eps}\tilde \rho+\rho_{\eps}\partial_{t}\tilde \rho\bigr)\d x+\int_{\T^{d}}\tilde \rho\partial_{t}\tilde \rho\d x \\
			&=\int_{\T^{d}}\int_{\R^{d}}\frac{1}{\eps^{2}}h_{\eps}\rho_{\eps}^{\alpha}\bigl(\rho_{\eps}-h_{\eps}\bigr)\mathcal{M}\d v\d x+\int_{\T^{d}}\frac{1}{\eps}\nabla_{x}\cdot j_{\eps}\tilde \rho\d x \\
			&\qquad \quad -\int_{\T^{d}}\rho_{\eps}\nabla_{x}\cdot\bigl(\tilde \rho^{-\alpha}\nabla_{x}\tilde \rho\bigr)\d x+\int_{\T^{d}}\tilde \rho\nabla_{x}\cdot\bigl(\tilde \rho^{-\alpha}\nabla_{x}\tilde \rho\bigr)\d x \\
			&=-\int_{\T^{d}}\int_{\R^{d}}\frac{1}{\eps^{2}}\rho_{\eps}^{\alpha}\bigl(h_{\eps}-\rho_{\eps}\bigr)^{2}\mathcal{M}\d v\d x-\int_{\T^{d}}\frac{1}{\eps}j_{\eps}\cdot \nabla_{x}\tilde \rho\d x \\
			&\qquad \quad +\int_{\T^{d}}\tilde \rho^{-\alpha}\nabla_{x}\rho_{\eps}\cdot\nabla_{x}\tilde \rho\d x-\int_{\T^{d}}\tilde \rho^{-\alpha}|\nabla_{x}\tilde \rho|^{2}\d x \\
			&\leq-\int_{\T^{d}}\frac{1}{\eps^{2}}\rho_{\eps}^{\alpha}|j_{\eps}|^{2}\d x-\int_{\T^{d}}\tilde \rho^{-\alpha}|\nabla_{x}\tilde \rho|^{2}\d x\\
			&\qquad \quad -\int_{\T^{d}}\frac{1}{\eps}j_{\eps}\cdot\nabla_{x}\tilde \rho\d x+\int_{\T^{d}}\tilde \rho^{-\alpha}\nabla_{x}\rho_{\eps}\cdot\nabla_{x}\tilde \rho\d x,
		\end{align*}
		where on the last line we applied Lemma \ref{lem:flux,energy-bounds}.
		
		Now, by Remark \eqref{rem:R_eps}, we expect formally that as $\eps \rightarrow 0$ we will have 
		\begin{align} \label{convergence}
			- \frac{1}{\eps} \rho_\eps^\alpha j_\eps - \nabla_x \rho_\eps \to  0.
		\end{align}
		Furthermore, we can see that
		\begin{align*}
			- \frac{1}{\eps} \rho_\eps^\alpha j_\eps - \nabla_x \rho_\eps = \eps \partial_t j_\eps + \nabla_x \cdot E_\eps. 
		\end{align*}
		This means that \eqref{convergence} does in fact hold in $H^{-1}_{t,x}([0,T] \times \mathbb{T}^d)$ as it follows from the convergence of $j_\eps, E_\eps$ to zero in $L^2([0,T] \times \mathbb{T}^d)$. We give a detailed proof of this later on. We aim to exploit this fact by rearranging the dissipation of the relative entropy. First, we complete the square with the two negative dissipation terms. To this end, we define
		\begin{align*}
			Q_{\eps}:=\frac{1}{\eps}\rho_{\eps}^{\frac{\alpha}{2}}j_{\eps}+\tilde \rho^{-\frac{\alpha}{2}}\nabla_{x}\tilde \rho, \quad \text{and} \quad R_{\eps}:=\nabla_{x}\rho_{\eps}+\frac{1}{\eps}\rho_{\eps}^{\alpha}j_{\eps}
		\end{align*}
		Then we have,
		\begin{align*}
			\frac{1}{2}\frac{\d}{\d t}\|h_{\eps}-\tilde \rho\|_{L^{2}_{x,v}(\mathcal{M})}^{2}&\leq-\int_{\T^{d}}|Q_{\eps}|^{2}\d x-\int_{\T^{d}}\biggl(1-2\Bigl(\frac{\rho_{\eps}}{\tilde \rho}\Bigr)^{\frac{\alpha}{2}}\biggr)\frac{1}{\eps}j_{\eps}\cdot\nabla_{x}\tilde \rho\d x+\int_{\T^{d}}\tilde \rho^{-\alpha}\nabla_{x}\rho_{\eps}\cdot\nabla_{x}\tilde \rho\d x
			\\
			&=-\int_{\T^{d}}|Q_{\eps}|^{2}\d x-\int_{\T^{d}}\biggl(1-2\Bigl(\frac{\rho_{\eps}}{\tilde \rho}\Bigr)^{\frac{\alpha}{2}}+\Bigl(\frac{\rho_{\eps}}{\tilde \rho}\Bigr)^{\alpha}\biggr)\frac{1}{\eps}j_{\eps}\cdot\nabla_{x}\tilde \rho\d x \\
			&\qquad \quad +\int_{\T^{d}}\Bigl(\nabla_{x}\rho_{\eps}+\frac{1}{\eps}\rho_{\eps}^{\alpha}j_{\eps}\Bigr)\cdot\tilde \rho^{-\alpha}\nabla_{x}\tilde \rho\d x \\
			&=-\int_{\T^{d}}|Q_{\eps}|^{2}\d x-\int_{\T^{d}}\biggl(1-\Bigl(\frac{\rho_{\eps}}{\tilde \rho}\Bigr)^{\frac{\alpha}{2}}\biggr)^{2}\frac{1}{\eps}j_{\eps}\cdot\nabla_{x}\tilde \rho\d x+\int_{\T^{d}}R_{\eps}\cdot\tilde \rho^{-\alpha}\nabla_{x}\tilde \rho\d x,
		\end{align*}
		In order to deal with the remaining $\eps^{-1}$ term we rewrite $\eps^{-1}j_{\eps}=\rho_{\eps}^{-\frac{\alpha}{2}}Q_{\eps}-(\rho_{\eps}\tilde \rho)^{-\frac{\alpha}{2}}\nabla_{x}\tilde \rho$ and substitute in to find
		\begin{align}
			\frac{1}{2}\frac{\d}{\d t}\|h_{\eps}-\tilde \rho\|_{L^{2}_{x,v}(\mathcal{M})}^{2}&\leq-\int_{\T^{d}}|Q_{\eps}|^{2}\d x-\int_{\T^{d}}\biggl(1-\Bigl(\frac{\rho_{\eps}}{\tilde \rho}\Bigr)^{\frac{\alpha}{2}}\biggr)^{2}\rho_{\eps}^{-\frac{\alpha}{2}}\nabla_{x}\tilde \rho\cdot Q_{\eps}\d x\nonumber\\
			&\qquad \quad +\int_{\T^{d}}\biggl(1-\Bigl(\frac{\rho_{\eps}}{\tilde \rho}\Bigr)^{\frac{\alpha}{2}}\biggr)^{2}(\rho_{\eps}\tilde \rho)^{-\frac{\alpha}{2}}|\nabla_{x}\tilde \rho|^{2}\d x+\int_{\T^{d}}R_{\eps}\cdot\tilde \rho^{-\alpha}\nabla_{x}\tilde \rho\d x \nonumber\\
			&\leq\int_{\T^{d}}R_{\eps}\cdot\tilde \rho^{-\alpha}\nabla_{x}\tilde \rho\d x+\frac{1}{4}\int_{\T^{d}}\biggl(1-\Bigl(\frac{\rho_{\eps}}{\tilde \rho}\Bigr)^{\frac{\alpha}{2}}\biggr)^{4}\rho_{\eps}^{-\alpha}|\nabla_{x}\tilde \rho|^{2}\d x \nonumber\\
			&\qquad \quad +\int_{\T^{d}}\biggl(1-\Bigl(\frac{\rho_{\eps}}{\tilde \rho}\Bigr)^{\frac{\alpha}{2}}\biggr)^{2}(\rho_{\eps}\tilde \rho)^{-\frac{\alpha}{2}}|\nabla_{x}\tilde \rho|^{2}\d x \nonumber\\
			&\leq\int_{\T^{d}}R_{\eps}\cdot\tilde \rho^{-\alpha}\nabla_{x}\tilde \rho\d x+\frac{1}{4}\int_{\T^{d}}\biggl(1-\Bigl(\frac{\rho_{\eps}}{\tilde \rho}\Bigr)^{\alpha}\biggr)^{2}\rho_{\eps}^{-\alpha}|\nabla_{x}\tilde \rho|^{2}\d x\nonumber \\
			&\qquad \quad +\int_{\T^{d}}\biggl(1-\Bigl(\frac{\rho_{\eps}}{\tilde \rho}\Bigr)^{\frac{\alpha}{2}}\biggr)^{2}(\rho_{\eps}\tilde \rho)^{-\frac{\alpha}{2}}|\nabla_{x}\tilde \rho|^{2}\d x\nonumber \\
			&:=\int_{\T^{d}}R_{\eps}\cdot\tilde \rho^{-\alpha}\nabla_{x}\tilde \rho\d x+I_{1}+I_{2}. \label{eq:diss-pre-gronwall}
		\end{align}
		In the above calculation, we first applied Young's inequality to deal with the $Q_{\eps}$ terms, followed by the inequality $|f^{\frac{1}{2}}-g^{\frac{1}{2}}|\leq|f-g|^{\frac{1}{2}}$ for all $f,g\geq0$. 
		
		We now wish to bound the integrals $I_{1}, I_{2}$ in terms of the distance $\|\rho_{\eps}-\tilde \rho\|_{L_{x}^{2}}$ using the uniform in $\eps$ and time bounds on $\rho_{\eps},\tilde \rho$, as well as the regularity of the limit equation.
		
		We begin by bounding the first integral term $I_{1}$. H\"older's inequality gives us
		\begin{equation} \label{eq:I1-inequality}
			I_{1}\leq\frac{1}{4}\bigl\|\rho_{\eps}^{-\alpha}|\nabla_{x}\tilde \rho|^{2}\bigr\|_{L_{t,x}^{\infty}}\int_{\T^{d}}\biggl(1-\Bigl(\frac{\rho_{\eps}}{\tilde \rho}\Bigr)^{\alpha}\biggr)^{2}\d x\lesssim\Bigl\|1-\Bigl(\frac{\rho_{\eps}}{\tilde \rho}\Bigr)^{\alpha}\Bigr\|_{L^{2}_{x}}^{2},
		\end{equation}
		with the constant depending only on $A,\alpha$ and $\|\rho_{\mathrm{in}}\|_{C^{2}}$ thanks to Lemma \ref{lem:rho_well-posedness}. We now apply Taylor's theorem to the function $f(z)=z^{\alpha}$ at $a=1$ to write $z^{\alpha}=1+\alpha(z-1)+h(z)(z-1)$ for some smooth $h:\R\rightarrow\R$ with $\lim_{z\rightarrow1}h(z)=0$. Thus, for any compact interval $I$ we have $|1-z^{\alpha}|\lesssim|1-z|$ for all $z\in I$ with the constant depending only on $\alpha, |I|$. Thanks to the uniform bounds on $\rho_{\eps},\tilde \rho$ we may apply this to \eqref{eq:I1-inequality} to find
		\begin{equation*}
			I_{1}\lesssim \Bigl\|1-\frac{\rho_{\eps}}{\tilde \rho}\Bigr\|_{L^{2}_{x}}^{2}\lesssim\int_{\T^{d}}\bigl(\rho_{\eps}-\tilde \rho\bigr)^{2}\d x=\int_{\T^{d}}\biggl(\int_{\R^{d}}\bigl(h_{\eps}-\tilde \rho\bigr)\mathcal{M}\d v\biggr)^{2}\d x\leq \|h_{\eps}-\tilde \rho\|_{L_{x,v}^{2}(\mathcal{M})}^{2},
		\end{equation*}
		where we again used the uniform bounds on $\tilde \rho$. Then, the final inequality follows from Jensen's inequality. We can bound $I_{2}$ in a similar way with a constant depending on the same quantities.
		
		Plugging these estimates back in to \eqref{eq:diss-pre-gronwall} and using Gr\"onwall's lemma, we find that for any $T>0$ we have
		\begin{equation}\label{eq:diss-post-gronwall}
			\|h_{\eps}-\tilde \rho\|_{L^{2}_{x,v}(\mathcal{M})}^{2}\leq e^{Ct}\lp \|h_{\mathrm{in}}-\rho_{\mathrm{in}}\|_{L^{2}_{x,v}(\mathcal{M})}^{2}+2\int_{0}^{t}e^{-Cs}\int_{\T^{d}}R_{\eps}\cdot \tilde \rho^{-\alpha} \nabla_{x}\tilde \rho\d x\d s\rp,
		\end{equation}
		for some $C>0$ depending only on $\alpha,A,\|\rho_{\mathrm{in}}\|_{C^{2}}$, and so in order to conclude, it remains only to bound the final term. 
		
		Note that, using the definition $j_\eps$ in \eqref{hydro_quantitites}, we may express $R_{\eps}$ in terms of derivatives of hydrodynamic quantities, i.e., $R_{\eps}=-\nabla_{x}\cdot E_{\eps}-\eps\partial_{t}j_{\eps}$. Thanks to the smoothness in time and space of the limiting equation we may use $\tilde \rho^{-\alpha} \nabla_{x}\tilde \rho$ as a test function onto which we pass the derivatives appearing in $R_{\eps}$. Thus, integrating by parts and applying H\"older's inequality we have
		\begin{align*}
			\int_{0}^{t}\int_{\T^{d}}e^{-Cs}R_{\eps}\cdot \tilde \rho^{-\alpha} \nabla_{x}\tilde \rho\dx\d s&=\int_{0}^{t}\int_{\T^{d}}\Bigl(e^{-Cs}E_{\eps}:\nabla_{x}\bigl(\tilde \rho^{-\alpha}\nabla_{x}\tilde \rho\bigr)+\eps j_{\eps}\cdot\partial_{s}\bigl(e^{-Cs}\tilde \rho^{-\alpha}\nabla_{x}\tilde \rho\bigr)\Bigr)\dx \d s \\
			&\qquad \quad -\eps\int_{\T^{d}}\Bigl(e^{-Ct}j_{\eps}(t)\cdot\tilde \rho^{-\alpha}(t)\nabla_{x}\tilde \rho(t)-j_{\eps}(0)\cdot\tilde \rho^{-\alpha}(0)\nabla_{x}\tilde \rho(0)\Bigr)\d x \\
			&\leq\Bigl\|\nabla_{x}\bigl(\tilde \rho^{-\alpha}\nabla_{x}\tilde \rho\bigr)\Bigr\|_{L^{\infty}_{t,x}}\int_{0}^{t}\int_{\T^{d}}|E_{\eps}|\d x\d s\\
			&\qquad \quad +\eps\Bigl(C\|\tilde \rho^{-\alpha}\nabla_{x}\tilde \rho\|_{L^{\infty}_{t,x}}+\bigl\|\partial_{t}\bigl(\tilde \rho^{-\alpha}\nabla_{x}\tilde \rho\bigr)\bigr\|_{L^{\infty}_{t,x}}\Bigr)\int_{0}^{t}\int_{\T^{d}}|j_{\eps}|\d x\d s \\
			&\qquad\qquad\qquad+2\eps\|\tilde \rho^{-\alpha}\nabla_{x}\tilde \rho\|_{L^{\infty}_{t,x}}\|j_{\eps}\|_{L^{\infty}(\R_{+};L_{x}^{1})}\\
			&\lesssim\eps\int_{0}^{t}\int_{\T^{d}}|j_{\eps}|\d x\d s+\int_{0}^{t}\int_{\T^{d}}|E_{\eps}|\d x\d s+\eps \\
			&\leq (1+\eps)\int_{0}^{t}\int_{\T^{d}}\biggl(\int_{\R^{d}}\bigl(h_{\eps}-\rho_{\eps}\bigr)^{2}\mathcal{M}\d v\biggr)^{\frac{1}{2}}\d x\d s+\eps,
		\end{align*}
		where we used Lemma \ref{lem:flux,energy-bounds} to reach the final inequality. Finally, we use Jensen's inequality once more to bound the remaining integral by
		\begin{equation*}
			\int_{0}^{t}\int_{\T^{d}}\biggl(\int_{\R^{d}}\bigl(h_{\eps}-\rho_{\eps}\bigr)^{2}\mathcal{M}\d v\biggr)^{\frac{1}{2}}\d x\d s\leq t^{\frac{1}{2}}\biggl(\int_{0}^{t}\|h_{\eps}-\rho_{\eps}\|_{L^{2}_{x,v}(\mathcal{M})}^{2}\d s\biggr)^{\frac{1}{2}}.
		\end{equation*}
		Lastly, we have from the proof of Theorem \ref{thm:longtime},
		\begin{align*}
			\frac{\mathrm{d}}{\mathrm{d}t}\|h_\eps \|_{L^{2}_{x,v}(\mathcal{M}) }\lesssim -  \frac{1}{\eps^2} \| h_\eps-\rho_\eps\|_{L^{2}_{x,v}(\mathcal{M}) },
		\end{align*}
		and so we have that
		\[ \int_0^t \|h_\eps(s)-\rho_\eps(s)\|_{L^{2}_{x,v}(\mathcal{M}) }^2 \mathrm{d}s \lesssim \eps^2 \|h_\eps(0) \|^2_{L^{2}_{x,v}(\mathcal{M}) } \lesssim \eps^2. \]
		Hence,
		\begin{equation*}
			\int_{0}^{t}\int_{\T^{d}}e^{-Cs}R_{\eps}\cdot\tilde \rho^{-\alpha}\nabla_{x}\tilde \rho\d s\d x\lesssim t^{\frac{1}{2}}\bigl(\eps+\eps^{2}\bigr)+\eps\lesssim \eps\bigl(1+t^{\frac{1}{2}}\bigr),
		\end{equation*}
		for $\eps\in(0,1)$. Plugging this back into \eqref{eq:diss-post-gronwall} we have the result.
	\end{proof}
	
	With this result, we are ready to show the global in time diffusive asymptotics. Recall that in the following we take $0<\eps, \eps'< \frac{1}{2}$, where $\eps$ denotes our scaling parameter and $\eps'$ is a bound on the distance between the initial data for the kinetic and parabolic equations, i.e. $\|f_{\eps,\mathrm{in}}-\rho_{\mathrm{in}}\mathcal{M}\|_{L_{x,v}^{2}(\mathcal{M}^{-1})}\leq\eps'$.
	
	\medskip 
	\noindent\textit{Proof of Theorem \ref{thm:asymptotics}.} Let $T>0$. By Proposition \ref{prp:short-time} there exist constants $C,C'>0$ such that for all $0\leq t\leq T$ we have
	\begin{equation*}
		\|f_{\eps}-\tilde \rho\mathcal{M}\|_{L_{x,v}^{2}(\mathcal{M}^{-1})}^{2}\leq e^{CT}\Bigl((\eps')^{2}+C'\eps\bigl(1+T^{\frac{1}{2}}\bigr)\Bigr)\lesssim e^{CT}\bigl(\eps+\eps'\bigr)\Bigl(1+T^{\frac{1}{2}}\Bigr).
	\end{equation*}
	We could optimise our choice of $T>0$ but for simplicity we choose $T=-\kappa\log(\eps+\eps')$ with $\kappa:=\frac{1}{2C}$. With this choice of $T$ we find
	\begin{align*}
		\|f_{\eps}-\tilde \rho\mathcal{M}\|^{2}_{L_{x,v}^{2}(\mathcal{M}^{-1})}&\lesssim(\eps+\eps')^{\frac{1}{2}}\Bigl(1+(-\kappa\log\bigl(\eps+\eps')\bigr)^{\frac{1}{2}}\Bigr) \\
		&\leq(\eps+\eps')^{\frac{1}{4}}\Bigl(1+(\eps+\eps')^{\frac{1}{4}}(-\kappa\log\bigl(\eps+\eps')\bigr)^{\frac{1}{2}}\Bigr) \\
		&\lesssim (\eps+\eps')^{\frac{1}{4}},
	\end{align*}
	where we used the boundedness of $x^{\frac{1}{2}}\log x$ on $(0,1)$ to bound the $\log(\eps+\eps')$ term appearing on the right-hand-side by a power law. Regarding the long-time asymptotics, thanks to Proposition \ref{prp:long_time_asymptotics} there exists $\tilde{C}>0$ such that for all $t\geq T$ we have
	\begin{equation*}
		\|f_{\eps}-\tilde \rho\mathcal{M}\|_{L_{x,v}^{2}(\mathcal{M}^{-1})}\lesssim e^{-\tilde{C}T}=(\eps+\eps')^{\tilde{C}\kappa},
	\end{equation*}
	with $\tilde{C}$ depending on $\alpha, A$ and $d$ and this completes the proof. 
	\qed

	\section*{Acknowledgements}
	The authors would like to thank Nikita Simonov for interesting discussions around the beginning of this project.
	
	\noindent
	J. Evans is supported by a Royal Society University Research Fellowship R1\_251808 and before this a Leverhulme early career fellowship ECF-2021-134. H. Yolda\c{s} is supported by the Dutch Research Council (NWO) under the NWO-Talent Programme Veni ENW project MetaMathBio with the project number VI.Veni.222.288. Some part of this research was conducted while H. Yolda\c{s} was visiting the Okinawa Institute of Science and Technology (OIST) through the Theoretical Sciences Visiting Program (TSVP).  For the purpose of open access, the authors have applied a Creative Commons Attribution (CC-BY) licence to any Author Accepted Manuscript version arising from this submission.

	\bibliography{nonlinearkinetic-nonlineardiffusion}

\end{document}